\documentclass[a4paper,12pt,leqno]{article}

\usepackage{amsmath,amssymb,latexsym,amsfonts,amsthm}

\theoremstyle{plain}
\newtheorem{Thm}{Theorem}[section]

\newtheorem{Lem}[Thm]{Lemma}
\newtheorem{Prop}[Thm]{Proposition}
\newtheorem{Cor}[Thm]{Corollary}

\theoremstyle{definition}

\newcommand{\bR}{\ensuremath{\mathbb{R}}}

\newcommand{\cF}{\ensuremath{\mathcal{F}}}

\newcommand{\cP}{\ensuremath{\mathcal{P}}}

\newcommand{\eps}{\ensuremath{\varepsilon}}

\renewcommand{\d}{{\rm d}}

\newcommand{\law}{\stackrel{{\rm law}}{=}}

\newcommand{\abra}[1]{\left| #1 \right|}

\newcommand{\kbra}[1]{\left\{ #1 \right\}}
\newcommand{\ebra}[1]{\left[ #1 \right]}

\numberwithin{equation}{section}

\newcounter{No}

\newcounter{Ci}[subsection]
\renewcommand{\theCi}{\arabic{Ci}$^{\circ}$}
\newcommand{\Circ}{\noindent \refstepcounter{Ci} {\bf \theCi).} }

\makeatletter
\renewcommand\section{\@startsection {section}{1}{\z@}%
                                   {-3.5ex \@plus -1ex \@minus -.2ex}%
                                   {2.3ex \@plus.2ex}%
                                   {\normalfont\large\bf}}
\makeatother

\begin{document}

\begin{center}
{\Large \bf Remarks on the density of the law of the occupation time 
for Bessel bridges and stable excursions}
\end{center}
\begin{center}
Kouji \textsc{Yano}\footnote{
E-mail: {\tt kyano@cr.math.sci.osaka-u.ac.jp}, 
Department of Mathematics, Osaka University, Osaka, Japan. 
Supported by JSPS Research Fellowships for Young Scientists.}
\qquad and \qquad 
Yuko \textsc{Yano}\footnote{
E-mail: {\tt yyano@fc.ritsumei.ac.jp}, 
Research Organization of Social Sciences, Ritsumeikan University, Kusatsu Shiga, Japan.} 
\end{center}
\begin{center}
{\small \today}
\end{center}
\bigskip



\begin{abstract}
Smoothness and asymptotic behaviors are studied 
for the densities of the law of the occupation time on the positive line 
for Bessel bridges and the normalized excursion of strictly stable processes. 
The key role is played by these properties 
for functions defined by Riemann--Liouville fractional integrals. 
\end{abstract}

\section{Introduction}

For a standard one-dimensional Brownian motion $ X=(X(t)) $, 
the occupation time $ A_+ = \int_0^1 1_{\{ X(s) > 0 \}} \d s $ 
has a density $ f_+(x) = 1/ \{ \pi \sqrt{x(1-x)} \} $, 
which is well-known as {\em L\'evy's arc-sine law}. 
In his formula, we find out that the density $ f_+(x) $ 
is smooth in $ (0,1) $ 
and diverges at $ x = 0+ $ as $ f_+(x) \sim (1/\pi) x^{-1/2} $. 
This implies the paradoxical fact that 
$ A_+ $ is more likely to take values near the extreme values $ 0 $ and $ 1 $ 
than near the central value $ 1/2 $. 

There have been a lot of attempts to study the law of the occupation time $ A_+ $ 
for other processes $ X $. 
For Bessel processes of dimension $ 0<d<2 $, 
Barlow--Pitman--Yor \cite{BPY} have found out that 
the occupation time $ A_+ $ has the same law as that introduced by Lamperti \cite{L}, 
and hence we see that the law has a density which is smooth in the interior of its domain 
and diverges at the origin of order $ x^{-d/2} $. 

For a Brownian bridge, 
it is also well-known as L\'evy's theorem 
that the law of the occupation time is uniform. 
For Bessel bridges, 
the second author \cite{Y} has obtained 
an expression of the distribution function 
in terms of the Riemann--Liouville fractional integral. 
We encounter a similar situation 
in the case of the normalized excursion of a strictly stable process 
(we call it a {\em stable excursion} in short) 
whose law has been characterized by Fitzsimmons--Getoor \cite{FG}. 
These expressions are so implicit 
that it is worth discussing 
regularity and asymptotic behaviors of their densities. 

In the present paper, 
we study the law of the occupation time 
for Bessel bridges and for stable excursions. 
We show that 
the distribution functions are smooth in the interior of the domain 
and determine the asymptotic behaviors of their densities at the extreme values. 
For these purposes, 
we study these properties 
for functions which are expressed by the Riemann--Liouville fractional integrals. 
Although the arguments are rather elementary, 
we could not find the results in the literature, 
and so it is still useful to the readers to deal with them 
to the extent which we need for our purposes. 

According to Theorem \ref{thm1} which is one of our main theorems, 
the law of the occupation time $ A_+ $ for a Bessel bridge $ X $ of dimension $ 0<d<2 $ 
has a smooth density $ f_+(x) $ 
such that $ f_+(x) \sim C x^{1-d} $ as $ x \to 0+ $ for some constant $ C $, 
which generalizes L\'evy's theorem for a Brownian bridge to Bessel bridges. 
We remark that the density $ f_+(x) $ diverges at $ x=0+ $ when $ 1<d<2 $ 
while it vanishes there when $ 0<d<1 $. 

The present paper is organized as follows. 
In Section \ref{sec: main}, we give a brief summary of the preceding results 
about the law of the occupation time for several processes. 
We also state our main theorems there. 
In Section \ref{sec: frac}, we give the precise definition 
of the Riemann--Liouville fractional integrals 
for a certain class of functions. 
We prove smoothness of our distribution functions in the end of that section. 
Section \ref{sec: asymp} 
is devoted to the proof of asymptotic behaviors of our distribution functions.

\section{Backgrounds and Main theorems} \label{sec: main}

\setcounter{Ci}{0}
Throughout the present paper, 
we suppose that the process considered 
which we will denote by $ X=(X(t):0 \le t \le 1) $ 
takes values in $ \bR $ and starts from 0 
and we will denote its occupation time up to time $ 1 $ on the positive line by 
$ A_+ = A_+(1) = \int_0^1 1_{\{ X(s)>0 \}} \d s $. 
By a density of the law of $ A_+ $ 
we always mean one with respect to the Lebesgue measure 
and will denote it by $ f_+ $. 
For two functions $ f,g $ defined on $ (0,\eps) $ for some $ \eps>0 $ 
such that $ f $ and $ g $ does not vanish on $ (0,\eps) $, 
we say that $ f(x) \sim g(x) $ if $ \lim_{x \to 0+} f(x)/g(x) = 1 $. 

\

\Circ 
When $ X $ is a standard Brownian motion, 
L\'evy's arc-sine law (see, e.g., \cite[pp. 57]{IM}) asserts that 
$ P(A_+ \le x) = \frac{2}{\pi} \sin^{-1} \sqrt{x} $, $ 0 \le x \le 1 $, i.e., 
\begin{align}
f_+(x) = \frac{1}{\pi \sqrt{x(1-x)}} 
, \qquad 0<x<1. 
\label{}
\end{align}
Here we remark that {\em (i) the density $ f_+ $ is smooth in $ (0,1) $}, 
and {\em (ii) it has an asymptotic $ f_+(x) \sim \frac{1}{\pi} x^{-1/2} $ as $ x \to 0+ $}. 

When $ X $ is a skew Bessel processes 
of dimension $ 0<d<2 $ with skewness parameter $ 0<p<1 $, 
Barlow--Pitman--Yor \cite{BPY} (see also Watanabe \cite{W}) 
have found out that $ A_+ $ has the same law 
as a random variable $ Y_{\alpha ,p} $ 
for $ \alpha = 1-d/2 \in (0,1) $ 
whose law is characterized by the Stieltjes transform: 
\begin{align}
E \ebra{ \frac{1}{\lambda + Y_{\alpha ,p}} } 
= \frac{p(\lambda+1)^{\alpha -1} + (1-p) \lambda^{\alpha -1}}{p(\lambda+1)^{\alpha } + (1-p) \lambda^{\alpha }} 
, \qquad \lambda>0 . 
\label{}
\end{align}
The class $ Y_{\alpha ,p} $ of random variables 
was introduced by Lamperti \cite{L} 
as the possible limit distributions of the occupation time of random walks 
(see also Fujihara--Kawamura--the second author \cite{FKY}). 
By inverting the Stieltjes transform, 
we see that the law of $ Y_{\alpha ,p} $ 
has a density $ f_+=f_{\alpha ,p} $ which is given by 
\begin{align}
f_{\alpha ,p}(x) = \frac{\sin \alpha \pi}{\pi} 
\cdot \frac{p(1-p) x^{\alpha -1} (1-x)^{\alpha -1}}{p^2 (1-x)^{2\alpha } + (1-p)^2 x^{2\alpha } + 2p(1-p)x^{\alpha } (1-x)^{\alpha } \cos \alpha \pi} 
\label{}
\end{align}
for $ 0<x<1 $. 
Here we also remark that 
{\em (i) the density $ f_{\alpha ,p} $ is smooth in $ (0,1) $}, 
and {\em (ii) it has an asymptotic} 
\begin{align}
f_{\alpha ,p}(x) \sim \frac{\sin \alpha \pi}{\pi} \cdot \frac{1-p}{p} \cdot x^{\alpha -1} 
\qquad \text{\em as} \ x \to 0+ . 
\label{}
\end{align}

For general diffusion processes $ X $, 
Kasahara--the second author \cite{KY} have studied the asymptotic behavior 
of the distribution function of $ A_+ $ 
under a certain regular variation assumption on the speed measure at the origin. 
Motivated by the result, 
Watanabe--the authors \cite{WYY} have proved 
that the law of $ A_+ $ has a density $ f_+ $ in quite a general case 
and also proved that 
the density is continuous in $ (0,1) $ 
and has an asymptotic $ f_+(x) \sim f_{\alpha ,p}(x) $ as $ x \to 0+ $ 
for some $ \alpha ,p $ under the regular variation assumption. 

We also remark on a mysterious resemblance between the law of $ Y_{\alpha ,1/2} $ 
and that of $ J_{\alpha } $ for $ 0<\alpha <1 $ 
which is characterized by its Stieltjes transform as 
\begin{align}
E \ebra{ \frac{1}{\lambda + J_{\alpha }} } 
= \frac{\alpha }{1-\alpha } \cdot \frac{\lambda^{\alpha -1} - (1+\lambda)^{\alpha -1}}{(1+\lambda)^{\alpha } - \lambda^{\alpha }} 
, \qquad \lambda>0 
\label{}
\end{align}
and which has a density $ f_{J_{\alpha }} $ on $ (0,1) $ with respect to the Lebesgue measure 
given by 
\begin{align}
f_{J_{\alpha }}(x) = \frac{\sin \alpha \pi}{\pi} \cdot \frac{\alpha }{1-\alpha } \cdot 
\frac{x^{\alpha -1} (1-x)^{\alpha -1}}{x^{2\alpha } + (1-x)^{2\alpha } - 2 x^{\alpha } (1-x)^{\alpha } \cos \alpha \pi} 
\label{}
\end{align}
for $ 0<x<1 $. 
The class of random variables $ J_{\alpha } $ 
has been introduced by Bertoin--Fujita--Roynette--Yor \cite[Theorem 1.1]{BFRY} 
to characterize the L\'evy measure of 
the law of the duration of the excursion straddling an independent standard exponential time 
for a Bessel process of dimension $ d=2-2\alpha $. 
They also introduced a two-parameter family of random variables $ J_{\alpha ,\beta} $ 
such that $ J_{\alpha ,\alpha } \law J_{\alpha } $ 
and $ J_{\alpha ,1-\alpha } \law Z_{\alpha } $ where $ Z_{\alpha } $ will be defined below.

\

\Circ 
In what follows, by a bridge process we mean 
the process obtained 
from a process $ Y $ for which all points are regular 
by conditioning on $ Y(1)=0 $. 
We refer to \cite{FPY} for the precise definition. 

When $ X $ is the bridge process of a Brownian motion (or a Brownian bridge), 
L\'evy's theorem (see, e.g., \cite[pp. 58]{IM}) asserts that 
the law of $ A_+ $ is uniform on $ (0,1) $. 
We point out that the law of $ A_+ $ has a constant density, 
which is a completely different situation from that of a Brownian motion. 
This result has been generalized to 
the bridge process of a L\'evy process (or a L\'evy bridge) 
for which all points are regular, 
by Fitzsimmons--Getoor \cite{FG} and Knight \cite{Kn} independently 
who showed that the law of $ A_+ $ is uniform on $ (0,1) $. 

Contrary to the case of L\'evy bridges, 
we encounter a drastically different situation in the case of diffusion bridges. 
When $ X $ is the bridge process of a skew Bessel process (or a skew Bessel bridge) 
of dimension $ 0<d<2 $ with skewness parameter $ 0<p<1 $, 
the second author \cite[Theorem 3.1]{Y} has proved that 
the distribution function of $ A_+ $ 
coincides with $ G_{1-2/d,p}(x) $ 
where $ G_{\alpha ,p}(x) $ for $ 0<\alpha <1 $ is a function on $ (0,1) $ 
characterized by its {\em generalized Stieltjes transform of index $ \alpha $} as 
\begin{align}
\int_0^1 \frac{\d G_{\alpha ,p}(x)}{(\lambda + x)^{\alpha }} 
= \frac{1}{p(1+\lambda)^{\alpha } + (1-p) \lambda^{\alpha }} 
, \qquad \lambda>0 . 
\label{GST of G alpha p}
\end{align}
Inverting the transform in the formula \eqref{GST of G alpha p}, 
she obtained the following expression of $ G_{\alpha ,p} $ 
in terms of the Riemann--Liouville fractional integral (see \cite[Theorem 4.1]{Y}): 
\begin{align}
G_{\alpha ,p}(x) = \int_0^x (x-t)^{\alpha -1} g_{\alpha ,p}(t) \d t 
, \qquad 0 \le x \le 1 
\label{G alpha p}
\end{align}
where 
\begin{align}
g_{\alpha ,p}(t) = \frac{\sin \alpha \pi}{\pi} 
\cdot \frac{(1-p) t^{\alpha }}{p^2 (1-t)^{2\alpha } + (1-p)^2 t^{2\alpha } 
+ 2p(1-p)t^{\alpha }(1-t)^{\alpha } \cos \alpha \pi} 
\label{g alpha p}
\end{align}
for $ 0<t<1 $. 
From \eqref{G alpha p} and Lemma \ref{lem: I alpha prime}, 
it follows that the density $ f_+ $ is given by 
\begin{align}
G_{\alpha ,p}'(x) = \int_0^x (x-t)^{\alpha -1} g_{\alpha ,p}'(t) \d t 
, \qquad 0 < x < 1 . 
\label{G alpha p prime}
\end{align}
We point out here that 
the integral $ \int_0^x (x-t)^{\alpha -1} g_{\alpha ,p}''(t) \d t $ is meaningless 
because of the asymptotic $ g_{\alpha ,p}''(t) \sim C_1 t^{\alpha -2} $ as $ t \to 0+ $ 
for some constant $ C_1 $. 
Nevertheless, the first one of our main theorems is 
\begin{Thm} \label{thm1}
The distribution function $ G_{\alpha ,p}(x) $ is infinitely differentiable in $ (0,1) $ 
and its derivative has an asymptotic 
\begin{align}
G_{\alpha ,p}'(x) \sim \frac{\sin \alpha \pi}{\pi} \cdot \frac{1-p}{p^2} \cdot 
\frac{\alpha \Gamma (\alpha )^2}{\Gamma (2 \alpha )} x^{2 \alpha -1} 
\qquad \text{as} \ x \to 0+ . 
\label{thm1-asymp}
\end{align}
\end{Thm}
The proof of Theorem \ref{thm1} will be given in Sections \ref{sec: frac} and \ref{sec: asymp}. 
We remark that 
the asymptotic behavior \eqref{thm1-asymp} of the density function $ G_{\alpha ,p}'(x) $ 
generalizes the asymptotic result \cite[pp.795]{Y} 
of the distribution function $ G_{\alpha ,p}(x) $.

\

\Circ 
When $ X $ is a L\'evy process such that $ P(X(t)>0) $ for $ t>0 $ is a constant $ 0<c<1 $, 
Getoor--Sharpe \cite{GS} has proved that 
$ A_+ $ has the same law as a random variable $ Z_c $ 
whose law has a density $ f_{Z_c} $ given by 
\begin{align}
f_{Z_c}(x) = \frac{\sin c \pi}{\pi} \cdot x^{c-1} (1-x)^{-c} 
, \qquad 0<x<1 . 
\label{generalized arcsine law}
\end{align}
It is proved by Getoor--Sharpe \cite{GS2} and Bertoin--Yor \cite{BY} 
that the occupation time of such a process $ X $ 
time-changed by the inverse local time of an independent Markov process 
has the same law as $ Z_c $. 

We encounter a drastically different situation again 
in the case of L\'evy excursions. 
Suppose that $ X $ obeys the conditional law of the excursion measure 
of an $ \alpha $-stable process of index $ 1<\alpha <2 $ 
given that the lifetime equals one. 
The law is that of the normalized excursion 
and is $ P^*_1 $ in the notation of Fitzsimmons--Getoor \cite{FG} 
to which we refer for the details. 
Then Fitzsimmons--Getoor \cite[eq. (4.26)]{FG} has proved 
the distribution function of $ A_+ $ 
which we denote by $ H_{1/\alpha }(x) $ 
where $ H_{\gamma } $ for $ 1/2<\gamma<1 $ is a function 
characterized by its generalized Stieltjes transform of index $ \gamma-1 $ as 
\begin{align}
\int_0^1 (\lambda + x)^{1-\gamma } \d H_{\gamma }(x) 
= \frac{\gamma }{(\lambda+1)^{\gamma } - \lambda^{\gamma }} 
, \qquad \lambda>0 . 
\label{GST of H gamma}
\end{align}
Integrating by parts in the LHS of \eqref{GST of H gamma} 
and then using the inversion formula \cite[Theorem 4.1]{Y}, 
we obtain 
\begin{align}
\int_0^x H_{\gamma }(y) \d y = \int_0^x (x-t)^{\gamma -1} h_{\gamma }(t) \d t 
\label{int H gamma}
\end{align}
where 
\begin{align}
h_{\gamma }(t) = \frac{\sin \gamma \pi}{\pi} \cdot \frac{\gamma}{1-\gamma } 
\cdot \frac{t^{\gamma }}{(1-t)^{2\gamma } + t^{2\gamma } 
- 2t^{\gamma }(1-t)^{\gamma } \cos \gamma \pi} . 
\label{h gamma}
\end{align}
We remak again on a mysterious resemblance between $ G_{\alpha ,1/2} $ and $ H_{\alpha } $. 
From \eqref{int H gamma} and Lemma \ref{lem: I alpha prime}, 
it follows that the distribution function is given by 
\begin{align}
H_{\gamma }(x) = \int_0^x (x-t)^{\gamma -1} h_{\gamma }'(t) \d t 
, \qquad 0 < x < 1 . 
\label{H gamma}
\end{align}
We point out here again that 
the integral $ \int_0^x (x-t)^{\gamma -1} h_{\gamma }''(t) \d t $ is meaningless 
because of the asymptotic $ h_{\gamma }''(t) \sim C_2 t^{\gamma -2} $ as $ t \to 0+ $ 
for some constant $ C_2 $. 
Nevertheless, the second one of our main theorems is 
\begin{Thm} \label{thm2}
Let $ 1/2 < \gamma < 1 $. 
Then the function $ H_{\gamma }(x) $ is infinitely differentiable in $ (0,1) $. 
In particular, its derivative is given by 
\begin{align}
H_{\gamma }'(x) = \frac{1}{x} \int_0^x (x-t)^{\gamma -1} 
\kbra{ \gamma h_{\gamma }'(t) + t h_{\gamma }''(t) } \d t 
, \qquad 0 < x < 1 . 
\label{H gamma prime}
\end{align}
and it has an asymptotic 
\begin{align}
H_{\gamma }'(x) \sim \frac{\sin \gamma \pi}{\pi} 
\cdot \frac{\Gamma (\gamma +1)^2}{\Gamma (2 \gamma )} \cdot \frac{2 \gamma -1}{1-\gamma } 
\cdot x^{2 \gamma -2} 
\qquad \text{as} \ x \to 0+ . 
\label{}
\end{align}
\end{Thm}
The proof of Theorem \ref{thm2} will be given in Sections \ref{sec: frac} and \ref{sec: asymp}.

\section{Differentiability of Riemann--Liouville fractional integrals} \label{sec: frac}

We denote by $ C(0,1) $ (resp. $ C_b(0,1) $) 
the class of continuous (resp. and bounded) functions on $ (0,1) $, 
and by $ C^n(0,1) $ the class of $ n $-times differentiable functions. 
We denote by $ L^1(0,r) $ for $ 0<r<1 $ the class of Borel measurable functions on $ (0,1) $ 
which are integrable on $ (0,r) $ with respect to the Lebesgue measure, 
and define $ L^1(0,1-) = \cap_{0<r<1} L^1(0,r) $. 

For $ \alpha >0 $, 
we define a linear operator $ I^{\alpha } $ on $ C_b(0,1) $ by 
\begin{align}
I^{\alpha }[f](x) = \int_0^x \frac{(x-t)^{\alpha -1}}{\Gamma (\alpha )} f(t) \d t 
, \qquad 0<x<1 , \ f \in C_b(0,1) 
\label{}
\end{align}
where the integral $ \int_0^x \d t $ stands for $ \int_0^1 \d t 1_{(0,x)}(t) $ 
in the sense of the Lebesgue integral. 
We remark that, if $ \alpha $ is a positive integer $ \alpha = n $, 
then $ I^n[f] $ is the $ n $-th multiple integral of $ f $. 
It is obvious that 
\begin{align}
\int_0^r \abra{ I^{\alpha }[f](x) } \d x 
\le \frac{r^{\alpha }}{\Gamma (\alpha +1)} \int_0^r |f(x)| \d x 
, \qquad 0<r<1, \ f \in C_b(0,1) . 
\label{}
\end{align}
Thus the operator $ I^{\alpha } $ on $ C_b(0,1) $ 
extends to a linear operator on $ L^1(0,1-) $ 
which is continuous when restricted on $ L^1(0,r) $ for all $ 0<r<1 $. 
We will denote the extension by the same symbol $ I^{\alpha } $ 
and we call it the {\em Riemann--Liouville fractional integral of order $ \alpha $}. 
The following lemma asserts that 
$ I^{\alpha } $ preserves the space of continuous functions:  

\begin{Lem}
Let $ \alpha >0 $. 
For $ f \in C_b(0,1) $ and for $ 0<a<b<1 $, the inequality 
\begin{align}
|I^{\alpha }[f](x)| 
\le C^1_{\alpha ,a,b} \int_0^a |f(x)| \d x 
+ C^2_{\alpha ,a,b} \sup_{x \in [a,b]} |f(x)| 
, \qquad x \in (a,b) 
\label{conti lemma ineq}
\end{align}
holds 
where $ C^1_{\alpha ,a,b}=\max \{ (b-a)^{\alpha -1},a^{\alpha -1} \} / \Gamma (\alpha ) $ 
and $ C^2_{\alpha ,a,b}=(b-a)^{\alpha } / \Gamma (\alpha +1) $. 
In particular, if $ f \in L^1(0,1-) \cap C(0,1) $, then 
$ I^{\alpha }[f] $ possesses a continuous version on $ (0,1) $. 
\end{Lem}

\begin{proof}
The inequality \eqref{conti lemma ineq} is immediate from the following obvious inequality: 
\begin{align}
|I^{\alpha }[f](x)| 
\le& \int_0^a \frac{(x-t)^{\alpha -1}}{\Gamma (\alpha )} |f(t)| \d t 
+ \int_a^x \frac{(x-t)^{\alpha -1}}{\Gamma (\alpha )} |f(t)| \d t 
\label{}
\end{align}
for all $ x \in (a,b) $. 

Let $ f \in L^1(0,1-) \cap C(0,1) $ and let $ 0<a<b<1 $. 
Take a sequence $ f_n \in C_b(0,1) $ 
which approximates $ f $ in $ L^1 $ on the interval $ (0,a) $ 
and uniformly on the interval $ [a,b] $. 
Then, by the inequality \eqref{conti lemma ineq}, 
we see that $ I^{\alpha }[f_n] $ converges to $ I^{\alpha }[f] $ 
uniformly on each compact subset of $ (a,b) $, 
which shows that $ I^{\alpha }[f] $ possesses a continuous version on $ (a,b) $. 
Since $ a $ and $ b $ are arbitrary, we obtain the second assertion. 
\end{proof}

If $ f \in L^1(0,1-) \cap C(0,1) $, 
we always assume that $ I^{\alpha }[f] $ stands for its continuous version. 
We know the well-known identity 
\begin{align}
I^{\beta } [I^{\alpha }[f]] = I^{\beta +\alpha }[f] 
, \qquad \alpha ,\beta >0 , \ f \in L^1(0,1-) . 
\label{}
\end{align}
This fact implies the following immediately: 
{\em If $ f \in L^1(0,1-) $ and if $ \alpha >1 $, 
then $ I^{\alpha }[f] $ is differentiable 
and 
\begin{align}
\frac{\d}{\d x} I^{\alpha }[f](x) = I^{\alpha -1}[f](x) 
\label{ddx I alpha f well-known}
\end{align}
holds.} 
As a sufficient condition for differentiability of $ I^{\alpha }[f] $ 
which is valid for all $ \alpha >0 $, 
the following lemma is also immediate: 

\begin{Lem} \label{lem: I alpha prime}
Let $ \alpha >0 $. 
If $ f \in C^1(0,1) $ and if $ f' \in L^1(0,1-) $, 
then $ I^{\alpha }[f] \in C^1(0,1) $ 
and its derivative is given by 
\begin{align}
\frac{\d}{\d x} I^{\alpha }[f](x) 
= \frac{f(0+)}{\Gamma (\alpha )} x^{\alpha -1} 
+ I^{\alpha }[f'](x) 
, \qquad 0<x<1 . 
\label{I alpha prime}
\end{align}
\end{Lem}

\begin{proof}
Since $ f' \in L^1(0,1-) $, the right-hand limit $ f(0+) $ exists. 
Since $ f(x) = f(0+) + I^1[f'](x) $, we have 
\begin{align}
I^{\alpha }[f](x) 
=& f(0+) I^{\alpha }[1](x) + I^{\alpha +1}[f'](x) 
\\
=& \frac{f(0+)}{\Gamma (\alpha +1)} x^{\alpha } + \int_0^x I^{\alpha }[f'](t) \d t . 
\label{}
\end{align}
This proves \eqref{I alpha prime}. 
\end{proof}

We can apply Lemma \ref{lem: I alpha prime} 
to $ f=g_{\alpha ,p} $ (resp. $ f=h_{\gamma } $) 
which is introduced in \eqref{g alpha p} (resp. \eqref{h gamma}) 
and obtain \eqref{G alpha p prime} (resp. \eqref{H gamma}). 
However, the integrability assumption of $ f' $ at the origin is {\em not} satisfied 
by $ f=g_{\alpha ,p}' $ nor by $ f=h_{\gamma }' $ 
because $ g_{\alpha ,p}''(t) \sim C_1 t^{\alpha -2} $ 
and $ h_{\gamma }''(t) \sim C_2 t^{\gamma -2} $ as $ t \to 0+ $. 
We would like to relax the integrability assumption of $ f' $ at the origin 
for differentiability of $ I^{\alpha }[f] $. 

For $ f \in C^1(0,1) $, we define 
\begin{align}
(\delta f)(x) = x f'(x) 
, \qquad x \in (0,1) . 
\end{align}
Now we obtain the key proposition as follows: 

\begin{Prop} \label{prop: key1}
Let $ \alpha >0 $. 
Suppose that $ f \in C^1(0,1) $ 
and that $ \delta f \in L^1(0,1-) $. 
Then $ f \in L^1(0,1-) $ and $ I^{\alpha }[f] \in C^1(0,1) $. 
Moreover, the following relation holds: 
\begin{align}
\delta (I^{\alpha }[f]) 
= I^{\alpha }[\alpha f + \delta f] . 
\label{derivative}
\end{align}
\end{Prop}

\begin{proof}
For $ 0<r<1 $, we have 
\begin{align}
\int_0^r |f(x)| \d x 
=& \int_0^r \d x \abra{ f(r) - \int_x^r f'(t) \d t } 
\label{}
\\
\le& r|f(r)| + \int_0^r \abra{ \delta f(t) } \d t < \infty . 
\label{}
\end{align}
This proves that $ f \in L^1(0,1-) $. 

Set $ g(t)=tf(t) $. 
Then we have $ g \in L^1(0,1-) \cap C^1(0,1) $ by the assumptions 
and we have $ g'=f+\delta f $. 
Since $ g $ satisfies the assumptions of Lemma \ref{lem: I alpha prime}, 
we see that $ I^{\alpha }[g] \in C^1(0,1) $ 
and that 
\begin{align}
\frac{\d}{\d x} I^{\alpha }[g](x) = I^{\alpha }[f+\delta f](x) . 
\label{ddx I alpha g}
\end{align}

We note that 
\begin{align}
\alpha I^{\alpha + 1}[f] 
= \int_0^x \frac{(x-t)^{\alpha -1}}{\Gamma (\alpha )} (x-t) f(t) \d t 
= x I^{\alpha }[f](x) - I^{\alpha }[g](x) . 
\label{I alpha + 1}
\end{align}
Hence we have 
\begin{align}
I^{\alpha }[f](x) = \frac{1}{x} \kbra{ \alpha I^{\alpha + 1}[f](x) + I^{\alpha }[g](x) } . 
\label{}
\end{align}
Now we conclude that $ I^{\alpha }[f] \in C^1(0,1) $, 
and using \eqref{ddx I alpha f well-known}, \eqref{ddx I alpha g} and \eqref{I alpha + 1}, 
we obtain 
\begin{align}
\frac{\d}{\d x} I^{\alpha }[f](x) = \frac{1}{x} I^{\alpha }[\alpha f + \delta f] (x) . 
\label{}
\end{align}
This completes the proof. 
\end{proof}

Immediately from Proposition \ref{prop: key1}, we obtain the following 

\begin{Cor} \label{cor: key1}
Let $ \alpha >0 $. 
Suppose that $ f \in C^n(0,1) $ 
and $ \delta^n f \in L^1(0,1-) $ for all $ n \ge 1 $. 
Then $ I^{\alpha }[f] \in C^n(0,1) $ for all $ n \ge 1 $ 
and its $ n $-th derivative for each $ n \ge 1 $ 
is given by 
\begin{align}
\frac{\d^n}{\d x^n} I^{\alpha }[f](x) = \frac{1}{x^n} I^{\alpha }[p_n(\alpha + \delta)f](x) 
\label{Dn I alpha f}
\end{align}
where $ p_n(t) = t(t-1) \cdots (t-n+1) $. 
\end{Cor}

\begin{proof}
We prove the assertion by induction. 
The assertion for $ n=1 $ is nothing but Proposition \ref{prop: key1}. 
Suppose that 
we have $ I^{\alpha }[f] \in C^n(0,1) $ and \eqref{Dn I alpha f} 
for a fixed $ n \ge 1 $. 
We note that 
$ p_n(\alpha + \delta)f \in C^1(0,1) $ and $ \delta p_n(\alpha + \delta)f \in L^1(0,1-) $ 
by the assumptions that $ f \in C^n(0,1) $ and $ \delta^n f \in L^1(0,1-) $ for all $ n \ge 1 $. 
Hence it follows from Proposition \ref{prop: key1} that 
$ I^{\alpha }[p_n(\alpha + \delta)f] \in C^1(0,1) $ 
and we have 
\begin{align}
\delta I^{\alpha }[p_n(\alpha + \delta)f] = I^{\alpha }[(\alpha + \delta) p_n(\alpha + \delta)f] . 
\label{}
\end{align}
Differentiating both sides of \eqref{Dn I alpha f} in $ x $, we obtain 
\begin{align}
\frac{\d^{n+1}}{\d x^{n+1}} I^{\alpha }[f](x) 
=& \frac{1}{x^{n+1}} \kbra{ \delta I^{\alpha }[p_n(\alpha + \delta)f](x) 
- n I^{\alpha }[p_n(\alpha + \delta)f](x) } 
\\
=& \frac{1}{x^{n+1}} I^{\alpha }[(\alpha + \delta -n) p_n(\alpha + \delta)f](x) 
\\
=& \frac{1}{x^{n+1}} I^{\alpha }[p_{n+1}(\alpha + \delta)f](x) . 
\label{}
\end{align}
This shows that the assertion holds for $ n+1 $, 
which completes the proof. 
\end{proof}

Now let us prove smoothness of our distribution functions 
$ G_{\alpha ,p} $ and $ H_{\gamma } $ which has been introduced 
in \eqref{G alpha p} and \eqref{H gamma}. 
For this purpose, we introduce a class of functions $ \cF $. 
Define $ \cP $ by the class of functions on $ (0,1) $ 
which are linear combinations 
of $ t^{\beta } v(t) $ for some $ \beta >0 $ 
and some infinitely differentiable function $ v $ on $ (0,1) $ 
whose derivatives of all orders are bounded near $ t=0 $. 
Now define $ \cF $ by the class of functions on $ (0,1) $ 
spanned by functions $ g $ of the form $ g(t) = t^{\beta -1} u(t)/v(t) $ 
for some $ \beta>0 $ and $ u,v \in \cP $. 
Note that $ \cF $ is a linear subspace of $ L^1(0,1-) \cap C(0,1) $.

\begin{Lem} \label{lemma}
If $ g \in \cF $, then $ \delta g \in \cF $. 
In particular, if $ g \in \cF $, 
then $ I^{\alpha }[g] $ for $ \alpha >0 $ is infinitely differentiable on $ (0,1) $. 
\end{Lem}

\begin{proof}
Let $ g(t)=t^{\beta -1} u(t)/v(t) $ 
for some $ \beta >0 $ and $ u,v \in \cP $. 
Then 
\begin{align}
\delta g(t) 
= \frac{t^{\beta -1} \kbra{(\beta -1) u(t) v(t) + \delta u(t) v(t) - u(t) \delta v(t)}}{v(t)^2} . 
\label{delta gt}
\end{align}
We remark the following facts: 
(i) $ u \in \cP $ implies $ \delta u \in \cP $; 
(ii) $ u,v \in \cP $ implies $ uv \in \cP $. 
Hence, by \eqref{delta gt}, we see that $ \delta g \in \cF $. 

Now it follows by induction 
that $ \delta^n g \in \cF $ for $ n \ge 1 $. 
Therefore we obtain 
$ \delta^n g \in L^1(0,1-) $ for $ n \ge 1 $. 
This proves the second assertion by Corollary \ref{cor: key1}. 
\end{proof}

Now we obtain the following result, 
which is the former halves of the statements of Theorems \ref{thm1} and \ref{thm2}: 

\begin{Thm}
The distribution functions $ G_{\alpha ,p} $ and $ H_{\gamma } $ 
are infinitely differentiable in $ (0,1) $. 
In particular, the densities $ G_{\alpha ,p}' $ and $ H_{\gamma }' $ 
are given by \eqref{G alpha p prime} and \eqref{H gamma prime}, respectively. 
\end{Thm}

\begin{proof}
We may rewrite 
\eqref{G alpha p} and \eqref{int H gamma} as 
$ G_{\alpha ,p} = \Gamma (\alpha ) I^{\alpha }[g_{\alpha ,p}] $ 
and $ I^1[H_{\gamma }]=\Gamma (\gamma ) I^{\gamma }[h_{\gamma }] $, respectively. 
We can easily see that $ g_{\alpha ,p}, h_{\gamma } \in \cF $, 
and therefore we obtain $ G_{\alpha ,p}, H_{\gamma } \in C^n(0,1) $ for all $ n \ge 1 $ 
by Lemma \ref{lemma}. 

Using Lemma \ref{lem: I alpha prime}, 
we obtain the formulae 
\begin{align}
G_{\alpha ,p}' = \Gamma (\alpha ) I^{\alpha }[g_{\alpha ,p}'] 
\label{G alpha p prime 2}
\end{align}
and 
\begin{align}
H_{\gamma }=\Gamma (\gamma ) I^{\gamma }[h_{\gamma }'] . 
\label{}
\end{align}
Using Proposition \ref{prop: key1}, 
we obtain the formula 
\begin{align}
H_{\gamma }'(x) 
= \frac{1}{x} \Gamma (\gamma ) I^{\gamma }[\gamma h_{\gamma }' + \delta h_{\gamma }'](x) . 
\label{H gamma prime 2}
\end{align}
Now the proof is completed. 
\end{proof}

\section{Asymptotic behaviors of Riemann--Liouville fractional integrals} \label{sec: asymp}

In this section, we study an asymptotic property for a function 
expressed by the Riemann--Liouville fractional integral. 
For this purpose, we use Karamata's theory of regular variations; 
see, e.g., \cite{BGT} for the details. 

\begin{Prop} \label{prop: key2}
Let $ \alpha >0 $ and $ f \in L^1(0,1-) \cap C(0,1) $. 
Suppose that 
\begin{align}
f(x) \sim x^{\beta -1} K(x) 
\qquad \text{as} \ x \to 0+ 
\label{ass reg var}
\end{align}
for some $ \beta>0 $ and some slowly varying function $ K(x) $ at $ x=0 $. 
Then 
\begin{align}
I^{\alpha }[f](x) \sim \frac{\Gamma (\beta )}{\Gamma (\alpha + \beta )} 
x^{\alpha + \beta -1} K(x) 
\qquad \text{as} \ x \to 0+ . 
\label{}
\end{align}
\end{Prop}

\begin{proof}
Since $ f \in C(0,1) $ and $ f $ satisfies the assumption \eqref{ass reg var}, 
the integral 
\begin{align}
I^{\alpha }[f](x) = \int_0^x \frac{(x-t)^{\alpha -1}}{\Gamma (\alpha )} f(t) \d t 
\label{}
\end{align}
makes sense in the sense of the Riemann integral. 
Changing variables to $ s=t/x $, we have 
\begin{align}
I^{\alpha }[f](x) = x^{\alpha } \int_0^1 \frac{(1-s)^{\alpha -1}}{\Gamma (\alpha )} f(xs) \d s . 
\label{}
\end{align}
By the assumption \eqref{ass reg var}, 
it is obvious that 
\begin{align}
\frac{I^{\alpha }[f](x)}{x^{\alpha + \beta -1} K(x)} 
\to \frac{1}{\Gamma (\alpha )} \int_0^1 (1-s)^{\alpha -1} s^{\beta -1} \d s 
= \frac{\Gamma (\alpha + \beta )}{\Gamma (\beta )} 
\label{}
\end{align}
as $ x \to 0+ $. This completes the proof. 
\end{proof}

\begin{proof}[Proof of Theorems \ref{thm1} and \ref{thm2}]
It is easy to see by \eqref{g alpha p} that 
\begin{align}
g_{\alpha ,p}'(t) \sim 
\frac{\sin \alpha \pi}{\pi} \cdot \frac{1-p}{p^2} \cdot \alpha t^{\alpha } 
\qquad \text{as} \ t \to 0+ 
\label{}
\end{align}
and by \eqref{h gamma} that 
\begin{align}
\gamma h_{\gamma }'(t) + \delta h_{\gamma }'(t) 
\sim \frac{\sin \gamma \pi}{\pi} \cdot \frac{\gamma ^2 (2\gamma -1)}{1-\gamma } 
\cdot t^{\gamma -1} 
\qquad \text{as} \ t \to 0+ . 
\label{}
\end{align}
Now we apply Proposition \ref{prop: key2} 
to \eqref{G alpha p prime 2} and \eqref{H gamma prime 2}, 
we obtain the desired results. 
\end{proof}

{\bf Acknowledgements:} 
The present study started 
when the first author had a short stay at University of California, San Diego 
in the winter of 2006. 
He expresses his sincere thanks to Professor Patrick J. Fitzsimmons 
for the fruitful discussions and his hospitality during that stay. 
Both of the authors would like to express their hearty gratitude 
to Professor Marc Yor 
for the stimulating discussions and his hospitality 
during their short stay in Universit\'e Paris VI in May 2007.

\end{document}